\documentclass[12pt]{amsart}

\usepackage{amsfonts,amsmath,amssymb}

\usepackage{enumerate}

\numberwithin{equation}{section}

\newtheorem{theorem}{Theorem}[section]

\newtheorem{lemma}[theorem]{Lemma}

\newtheorem{example}[theorem]{Example}

\newtheorem{definition}{Definition}[section]

\newtheorem{corollary}[theorem]{Corollary}

\newcommand{\cl}[1]{\mathcal{#1}} %\cl A

\newcommand{\bb}[1]{\mathbb{#1}}

\newcommand{\sca}[1]{\left\langle#1\right\rangle} %\sca{x,y}

\newcommand{\nor}[1]{\left\Vert #1\right\Vert}

\begin{document}

%\title[smaller title]{Stable isomorphism and strong Morita equivalence of operator algebras }

\title{ Stable isomorphism and strong Morita equivalence of operator algebras }

\author[G.~K.~Eleftherakis ]{G. K. Eleftherakis }

\address{Department of Mathematics\\Faculty of Sciences\\
University of Patras\\265 04 Patras, Greece }

\email{gelefth@math.upatras.gr}

%\keywords{TRO, stable isomorphism}

%\subjclass[2000]{Primary 46L15; Secondary 47L25}

\date{}

\maketitle

\begin{abstract} 
We introduce a Morita type equivalence: two operator algebras $A$ and $B$ are 
called strongly $\Delta $-equivalent if they have 
 completely isometric representations $\alpha $ and $\beta $ respectively and there exists a ternary ring of operators 
$M$ such that $\alpha (A)$ (resp. $\beta (B)$ ) is equal to the norm closure of the linear span of the set 
$M^*\beta (B)M, $ (resp. $M\alpha (A)M^*$). We study the properties of this equivalence. We prove 
that if two operator algebras $A$ and $B,$ 
 possessing countable  approximate identities, are strongly $\Delta $-equivalent, then the 
operator algebras $A\otimes \cl K$ and $B\otimes \cl K$ 
are isomorphic. Here $\cl K$ is the set of compact operators on an infinite dimensional 
separable Hilbert space and $\otimes $ is 
the spatial tensor product. Conversely, if $A\otimes \cl K$ and $B\otimes \cl K$ 
are isomorphic and $A, B$ possess contractive approximate identities then  
 $A$ and $B$ are strongly $\Delta $-equivalent. 
\end{abstract}

\section{Introduction}

An operator algebra $A$ is both an operator space and a Banach algebra for which there exists 
a Hilbert space $H$ and a completely isometric homomorphism $\alpha : A\rightarrow B(H)$,
where $B(H)$ is the set of bounded operators acting on $H$. If this algebra is a dual 
  space and the 
map $\alpha $ is weak* continuous,  it is called a dual operator algebra. 
 The topic of non-selfadjoint operator 
algebras, studied initially by Kadison, Singer, Ringrose and Arveson, has been fundamental 
for the theory of operator spaces.  

Rieffel introduced the notion of strong Morita equivalence of $C^*-$algebras and since then 
 many articles have been devoted to this topic. In \cite{bgr}, 
Brown, Green and Rieffel proved that two  $C^*-$algebras with countable approximate identities are 
strongly Morita equivalent if and only if they are strongly stably isomorphic. Blecher, 
Muhly and Paulsen introduced another concept of strong Morita equivalence for operator 
algebras, \cite{bmp}. In that article they proved  that their Morita equivalence doesn't 
induce a stable isomorphism between the operator algebras even if they possess an identity 
element of norm $1$. 

In the present article we construct a Morita type equivalence of operator algebras 
(strong $\Delta -$equivalence) and prove that if two operator algebras with countable  approximate identities 
are strongly $\Delta -$equivalent then they are strongly stably isomorphic. Conversely, if 
they are  strongly stably isomorphic and they possess contractive approximate identities, 
then they  are strongly $\Delta -$equivalent.

A fundamental tool in our theory is the concept  of a ternary ring of operators (TRO). A subspace 
$M$ of the set $B(H,K)$ of bounded operators from the Hilbert space $H$ to 
a Hilbert space $K$ is called a TRO if $MM^*M\subset M$.  In the Morita theory of  $C^*-$algebras, 
a TRO is an equivalence bimodule. In the case of  $\Delta -$equivalence, 
 the equivalence bimodules are ``generated'' by TROs.    

In \cite{ele1}, the notion of weak TRO equivalence was defined and its properties were studied in \cite{ele2, ele3} and \cite{elepaul}. It 
is important that the weak TRO equivalence of dual operator algebras is related to the notion of weak stable isomorphism. We recall 
some definitions and results from the above papers:

\begin{definition}\label{11} Suppose $A$ and $B$ are weakly* closed  algebras acting on the Hilbert spaces $H$ and $K$ respectively. They are said to be 
 weakly TRO equivalent if there exists a TRO $M\subset B(H, K)$ such that $$A=[M^*BM] ^{-w^*} \;\;\mbox{and}\;\;B=[MAM^*] ^{-w^*} .$$
\end{definition}

\begin{definition}\label{12} Suppose $A$ and $B$ are dual operator algebras. We call them weakly $\Delta $-equivalent if they have completely isometric 
normal representations 
 $\alpha $ and $\beta $ respectively such that $\alpha (A)$ and $\beta (B)$ are weakly TRO equivalent. 
\end{definition}

If two dual operator algebras are weakly $\Delta $-equivalent, then they are weakly Morita equivalent in the sense of \cite{bk, kash}. The 
converse does not hold, \cite{ele2, ele3, ele4}.

\begin{theorem}\label{13}\cite{elepaul} Two dual operator algebras $A$ and $B$ are weakly $\Delta $-equivalent iff there exists a cardinal $I$  
such that the dual operator algebras $A \otimes^\sigma  B(l^2(I))$ and $B \otimes^\sigma 
B(l^2(I))$ 
are isomorphic as dual operator algebras. Here $ \otimes^\sigma $ is the normal 
spatial tensor product. 
\end{theorem} 

A similar theorem for dual operator spaces is the main result of \cite{ept}.

In this paper we introduce the notion of strong TRO equivalence and of strong $\Delta $-equivalence:

\begin{definition}\label{14} Suppose $A$ and $B$ are norm  closed  algebras acting on the Hilbert spaces $H$ and $K$ respectively. We call them 
 strongly  TRO equivalent if there exists a TRO $M\subset B(H, K)$ such that $$A=[M^*BM] ^{-\|\cdot\|} \;\;\mbox{and}\;\;B=[MAM^*] ^{-\|\cdot
\|} .$$
\end{definition}

\begin{definition}\label{15} Suppose $A$ and $B$ are operator algebras. We call them strongly $\Delta $-equivalent if they have 
completely isometric representations 
 $\alpha $ and $\beta $ respectively such that $\alpha (A)$ and $\beta (B)$ are strongly  TRO equivalent. 
\end{definition}

In Section 2, we study some properties of Definitions \ref{14} and \ref{15} and we prove that both strong TRO equivalence and strong $\Delta $-equivalence 
are equivalence relations. We also prove that strong $\Delta $-equivalence is stronger than the BMP-strong Morita equivalence introduced in \cite{bmp}. 
 (In Section 3 we will see that strong $\Delta $-equivalence is strictly stronger than BMP-strong Morita equivalence). In Section 
2 we also prove that two $C^*-$algebras are strongly Morita equivalent in the sense of Rieffel \cite{strongrief} iff they are strongly $\Delta $-equivalent. 

In Section 3 we will prove that strong $\Delta $-equivalence is the appropriate context for 
the strong stable isomorphism of 
 operator algebras. Actually, generalising the results of 
\cite{bgr}, we will prove 
that if two operator algebras $A$ and $B$ with countable  approximate identities are strongly 
$\Delta $-equivalent, 
 then they are strongly stably isomorphic. 
This means that the algebras $A\otimes \cl K$ and $B\otimes \cl K$, where 
$\cl K$ is the algebra of compact operators 
acting on an infinite dimensional separable Hilbert space and $\otimes $ 
is the spatial tensor product, are isomorphic as operator 
spaces. Conversely, if  $A\otimes \cl K$ and $B\otimes \cl K$ are isomorphic and $A$ and $B$ 
possess contractive approximate identities,  then $A$ and $B$ are strongly 
$\Delta $-equivalent.

Throughout this paper, we will use the following lemma, which can be deduced from the proof of Theorem 6.1 of \cite{bmp}. 

\begin{lemma} \label{16} Suppose $M$ is a norm closed TRO. Then there exist nets $(u_t)_t, (f_\lambda )_\lambda $ where 
$$u_t=\sum_{i=1}^{l_t}(m_i^t)^*m_i^t, \;\;\;\;f_\lambda =\sum_{i=1}^{k_\lambda } n_i^\lambda (n_i^\lambda)^* $$
and $$\{m_i^t, n_j^\lambda : 1\leq i\leq l_t, \;\;1\leq j\leq k_\lambda \}\subset M$$ such that $$\|u_t\|\leq 1, \|f_\lambda \|\leq 1, 
\;\;\forall t, \lambda $$ and such that 
$$ \|\cdot\|-\lim _tu_tm^*=m^* , \;\;\;\|\cdot\|-\lim _\lambda f_\lambda m=m\;\;\;\forall m\in M. $$
\end{lemma}

A representation of an operator algebra $A$ is a completely contractive homomorphism 
$\alpha : A\rightarrow B(H)$ where $H$ is a Hilbert space. In case $A$ is a 
dual operator algebra, we call $\alpha $ a \textit{normal representation} of $A$ if it is  weakly* continuous.  
If $X$ is a right $A-$operator module and $Y$ is a left $A-$operator module over 
an operator algebra $A$, we denote by $X\otimes ^h_AY$ the $A$-balanced Haagerup tensor 
product of $X$ and $Y$ \cite{bmp}. This operator space has the property that it 
linearises the completely bounded $A$-balanced bilinear  maps $\phi : X\times Y\rightarrow Z$, where 
$Z$ is another operator space. The reader can use the books \cite{bmp, er, paul, p} for the 
notions and theorems of operator space theory which appear in this present paper.
If $X$ is a vector space, $M_{m,n}(X)$ denotes the set of $m\times n$ 
matrices with entries in $X$ and we write $M_n(X)$ for   $M_{n,n}(X)$,  
$C_n(X)$ for $M_{n,1}(X)$, and $R_n(X)$ for $M_{1,n}(X)$.

\section{Strong TRO equivalence and strong $\Delta $-equivalence}

\begin{theorem} \label{21} Strong TRO equivalence is an equivalence relation.
\end{theorem}
\begin{proof} If $A$ is an operator algebra acting on the Hilbert space $H,$  then 
$$A=M^*AM=MAM^*$$ where $M$ is the TRO $\bb{C}I_H$. So it suffices to prove the transitivity of strong TRO equivalence. 

Suppose $A, B$, and $C$ are operator algebras acting on the Hilbert spaces $H, K$, and $L$,  respectively, such that there exist 
TROs  $M\subset B(H, K)$ and $N\subset B(K, L)$ satisfying 
$$A=[M^*BM] ^{-\|\cdot\|} , \;\;B=[MAM^*] ^{-\|\cdot\|} =[N^*CN] ^{-\|\cdot\|} ,\;\;C=[NBN^*] ^{-\|\cdot\|} .$$ 
We have to show that $A$ and $C$ are strongly TRO equivalent. 

Let $D$ be the $C^*-$algebra generated by the sets  $MM^*$ and $N^*N$. Put $$T=[NDM] ^{-\|\cdot\|} \subset B(H,L).$$ 
We shall show that $T$ is a TRO implementing the TRO equivalence of $A$ and $C$. Firstly, we see that $T$ is a TRO: 
 Observe $$NDMM^*DN^*NDM\subset NDM\subset T.$$ Thus, $TT^*T\subset T.$ 
Now we have that $$TAT^*\subset [NDMAM^*DN^*]^{-\|\cdot\|} \subset [NDBDN^*]^{-\|\cdot\|} .$$
Since $$MM^*B\subset B, \;\;\;N^*NB\subset B,\;\;BMM^*\subset B, \;\;\;BN^*N\subset B,$$ and $D$ is generated by $MM^*$ and $N^*N$,  we have 
$$DBD\subset B.$$ Thus $$TAT^*\subset [NBN^*]^{-\|\cdot\|} \subset C.$$ On the other hand, 
\begin{align*} C=& [NBN^*] ^{-\|\cdot\|} =[NN^*NBNN^*N] ^{-\|\cdot\|} \subset \\& [NDBDN^*] ^{-\|\cdot\|} =
[NDMAM^*DN^*] ^{-\|\cdot\|}= [TAT^*]  ^{-\|\cdot\|}  .
\end{align*}
We have proved $$C=[TAT^*]  ^{-\|\cdot\|}  .$$ Similarly, we can prove that $$A=[T^*CT]  ^{-\|\cdot\|}  .$$ The 
proof is complete.
\end{proof}

\begin{theorem}\label{22} Suppose $A$ and $B$ are $C^*$-algebras. Then $A$ and $B$ are strongly $\Delta $-equivalent 
iff they are 
strongly 
Morita equivalent in the sense of Rieffel.
\end{theorem}
\begin{proof} Suppose that $A$ and $B$ are strongly Morita equivalent $C^*$-algebras in the sense of Rieffel. Then 
there exist faithful $*-$homomorphisms $\alpha $ of $A$ and $\beta $ of $B$ to $B(H)$ and $B(K)$, respectively, where $H$ and $ K$ are Hilbert 
spaces, and a TRO $M\subset B(H, K)$ such that 
$$\alpha(A)=[M^*M]   ^{-\|\cdot\|} , \;\;\; \beta (B)=[MM^*] ^{-\|\cdot\|} . $$ 
Now see that $$\beta (B)=[MM^*]  ^{-\|\cdot\|} =[MM^*MM^*]  ^{-\|\cdot\|} =[M\alpha (A)M^*]  ^{-\|\cdot\|} . $$ 
Similarly, we can prove that $\alpha (A)=[M^*\beta(B)M]  ^{-\|\cdot\|}.   $ 
For the converse, suppose  that $A$ and $B$ are $C^*$-algebras of operators and that there exists a TRO $M$ such that 
$$A=[M^*BM] ^{-\|\cdot\|} ,  \;\;\mbox{and}\;\;B=[MAM^*]  ^{-\|\cdot\|} .$$ 
Let $N= [BM] ^{-\|\cdot\|}.$ We have 
$NN^*N\subset [BMM^*BM] ^{-\|\cdot\|}.$ Since $MM^*M\subset M$, we have $MM^*B\subset B$ and thus $$NN^*N\subset [BM] ^{-\|\cdot\|} =N.$$
So $N$ is a TRO. We now see that
$$[N^*N]^{-\|\cdot\|} =[M^*BBM]  ^{-\|\cdot\|} =[M^*BM]  ^{-\|\cdot\|} =A,$$
$$[NN^*]^{-\|\cdot\|} =[BMM^*B]^{-\|\cdot\|} .$$ Since $M=[MM^*M]  ^{-\|\cdot\|} $, we have 
$$B=[MM^*MAM^*]  ^{-\|\cdot\|} =[MM^*B]  ^{-\|\cdot\|} .$$ 
So $$[NN^*]^{-\|\cdot\|} =[BB]  ^{-\|\cdot\|} =B.$$ Similarly we can prove $$A=[N^*N] ^{-\|\cdot\|} .$$
\end{proof}

\begin{theorem}\label{23} Suppose $A$ and $B$ are strongly 
TRO equivalent operator algebras acting on the Hilbert spaces $H$ and $K$, respectively.  Then their diagonals $\Delta (A)=A\cap A^*, \;\;\;\Delta (B)=B\cap B^*$ are strongly TRO 
equivalent.
\end{theorem}
\begin{proof} There exists a TRO $M\subset B(H, K)$ such that $$A=[M^*BM] ^{-\|\cdot\|} ,  \;\;\mbox{and}\;\;B=[MAM^*]  ^{-\|\cdot\|} .$$ 
Since $\Delta (A)$ and $\Delta (B)$ are $C^*-$algebras, we have 
$$ M^*\Delta (B)M\subset \Delta (A) , \;\;\;M\Delta (A)M^*\subset \Delta (B). $$ Suppose that $b\in \Delta (B)$. Let 
$(f_\lambda )$ be the net from Lemma \ref{16}. We have 
$\|\cdot\|-\lim_\lambda f_\lambda m=m$ for all $m\in M$. Since $B=[MAM^*] ^{-\|\cdot\|} $,  we have 
$$\|\cdot\|-\lim_\lambda f_\lambda b=b.$$ 
Also, since  $$\|\cdot\|-\lim_{\lambda^\prime }m^*f_{\lambda^\prime }^*=m^*\;\;\;\forall \;\;m\in M,$$
  we have $$\|\cdot\|-\lim_{\lambda^\prime }cf_{\lambda^\prime }^*=c\;\;\;\forall \;\;c\in B.$$
So $$\|\cdot\|-\lim_{\lambda^\prime } f_\lambda bf_{\lambda^\prime }^* =f_\lambda b.$$ 
But $$f_\lambda bf_{\lambda^\prime }^* \in [MM^*\Delta (B)MM^*]  ^{-\|\cdot\|} \subset [M\Delta (A)M^*]  ^{-\|\cdot\|} .$$
Thus $b\in [M\Delta (A)M^*]  ^{-\|\cdot\|} $. We have proved $\Delta (B)=[M\Delta (A)M^*]  ^{-\|\cdot\|} $. 
Similarly we can prove $\Delta (A)=[M^*\Delta (B)M]  ^{-\|\cdot\|} $. 

\end{proof}

\begin{corollary}\label{23a} Suppose $A$ and $B$ are operator algebras which are strongly 
$\Delta $-equivalent. Then their diagonals $\Delta (A)=A\cap A^*, \;\;\;\Delta (B)=B\cap B^*$ are strongly $\Delta $- 
equivalent.
\end{corollary}

\begin{theorem}\label{24} Suppose that $A$ and $B$ are strongly $\Delta $-equivalent operator algebras with contractive approximate identities (cai's).
Then $A$ and $B$ are strongly Morita equivalent in the sense of Blecher, Muhly and 
Paulsen, \cite{bmp}.
\end{theorem}
\begin{proof} Let $H$ and $ K$ be Hilbert spaces such that $A\subset B(H)$ and $ B\subset B(K)$. Assume that 
there exists a norm closed TRO $D\subset B(H, K)$ such that 
$$A=[D^*BD] ^{-\|\cdot\|}, \;\;\;B=[DAD^*] ^{-\|\cdot\|} .$$ 
Set
$$ U=[ B D]^{-\|\cdot\| }\;\;\; \text{and}\;\;\; V=[D^* B]^{-\|\cdot\| }.$$ 
Since $B D D^*\subset B$, we have 
$$ B D D^* B D\subset  B D\subset U\Rightarrow  U A\subset  U.$$
So $ U$ is a $ B-A$ bimodule. Similarly, we can prove that $ V$ is an $ A-B$ bimodule. 

Since 
$ D^* B B D\subset  A$, we have $ V U\subset  A$. The algebra $B$ has a cai, thus 
$$ B=\overline{ B B}^{\|\cdot\| }.$$ Therefore 
$$ A=[ D^* B D]^{-\|\cdot\|}= [ D^* B B D]^{-\|\cdot\|} \subset [ V U]^{-\|
\cdot\|}.$$ We have proved that $ A=[ V U]^{-\|\cdot\|}$. Similarly we can prove 
that  $ B=[ U V]^{-\|\cdot\|}$.  It now suffices to prove that $A$ (resp. $B$) is 
completely isometrically isomorphic with the space $ V \otimes ^h_{ B} U$ (resp. 
$U\otimes ^h_{ A}  V$ ).   

The completely contractive bilinear $B$-balanced $ A$-module map 
$$ V\times  U\rightarrow  A: (v,u)\rightarrow vu$$ 
induces a  completely contractive  $ A$-module map 
$$\theta : V\otimes ^h_{ B}  U\rightarrow  A: v\otimes_B u\rightarrow vu.$$
We shall prove that this map is isometric and onto.
Since  $A=[ V U]^{-\|\cdot\|}$, it suffices to prove that if $v\in R_k( V)$ 
and $u\in C_k(U)$, then $$\|v\otimes _{ B}u\|\leq \|vu\|.$$

Suppose that $v=(v_1,...,v_k)$. Since $ V=[ D^*B]^{-\|\cdot\|}$, there exist sequences $((\delta^i _n)^*)_n ,\;(b _n^i)_n,$ 
 where 
$$ (\delta^i _n)^*\in \; R_{l_n}( D^*), \;\;\; b _n^i\in \; C_{l_n}(B) $$ 
such that $$v_i=\|\cdot\|-\lim_n(\delta _n^i)^*b_n^i, 1\leq i\leq k.$$ If 
$$\delta _n^*=((\delta _n^1)^*,...,(\delta _n^k)^*), \;\;\;b_n=(b_n^1\oplus ...\oplus b_n^k),$$ we 
have  $$v=\|\cdot\|-\lim_n\delta _n^*b_n.$$  Thus 
 $$v\otimes  _{ B}u=\|\cdot\|-\lim_n\delta _n^*b_n\otimes _{ B}u, \;\;\;vu=\|\cdot\|-\lim_n\delta _n^*b_nu.$$   
Fix $\epsilon >0$. There exists $n$ such that 
$$\|v\otimes _{ B}u\|-\epsilon <\|\delta _n^*b_n\otimes _{ B}u\|-\frac{\epsilon }{2}$$
and $$\|\delta _n^*b_nu\|<\|vu\|+\epsilon .$$
By Lemma \ref{16}, there exists a net $ (d _m)_m$ where $ d_m\in  Ball( C_{k_m}(D) )$
 for all $m$ such that $$\|\cdot\|-\lim_md_m^*d_m\delta _n^*=\delta _n^*\;\;\forall \;\;n.$$
Therefore  $$\|\cdot\|-\lim_md_m^*d_m\delta _n^*b_n \otimes _{ B}u = \delta _n^*b_n\otimes _{ B}u  .$$
So there exists  $m$ such that 
$$ \|\delta _n^*b_n \otimes _{B}u  \|-\frac{\epsilon }{2} < \|d_m^*d_m\delta _n^*b_n \otimes _{\cl B}u  \|-
\frac{\epsilon }{4} . $$ 
Observe that $d_m\delta _n^*b_n$ is a matrix with entries in $B$. Since $B$ has a cai, there exists a net $(c_i)\subset Ball(B)$ 
such that $$\|\cdot\|-\lim_i  d_m^*(c_i\oplus ...\oplus c_i)d_m\delta _n^*b_n\otimes _Bu =d_m^*d_m\delta _n^*b_n\otimes _Bu .$$
So there exists  $i$ such that 
$$\|d_m^*d_m\delta _n^*b_n\otimes _Bu\|-\frac{\epsilon }{4} < \|d_m^*(c_i\oplus ...\oplus c_i) d_m\delta _n^*b_n\otimes _Bu\| .$$
Since $d_m\delta _n^*b_n$ is a matrix with entries in $B $ and the bilinear map 
$\otimes _B$ is $B-$balanced, we have 
\begin{align*}
 \|d_m^*(c_i\oplus ...\oplus c_i)d_m\delta _n^*b_n\otimes _Bu\| =& \|d_m^*(c_i\oplus ...\oplus c_i)\otimes _B d_m \delta _n^*b_nu\| \leq \\&
\|d_m^*(c_i\oplus ...\oplus c_i)\|\| d_m\delta _n^*b_nu\| .
\end{align*}
Since $$\|d_m\|\leq 1, \;\;\;\|c_i \|\leq 1,$$ we have 
\begin{align*}
\|d_m^*(c_i\oplus ...\oplus c_i) d_m\delta _n^*b_n\otimes _Bu\|\leq \|\delta _n^*b_nu\|<\|vu\| +\epsilon .
\end{align*}
We have proved that $$\|v\otimes _Bu\|-\epsilon <\|vu\|+\epsilon .$$
Since $\epsilon $ was arbitrary, we have $$ \|v\otimes _Bu\|\leq \|vu\| \Rightarrow  \|v\otimes _Bu\|=\|vu\| .$$
So $\theta $ is an isometry,  onto $A$. 

We need to show that the map $$ id_n\otimes \theta  : M_n( V \otimes ^h_{ B} U)\rightarrow M_n(A)$$ 
sending matrices of the form $(\sum_{i=1}^{n_{k,l}}v_i^{k,l}\otimes _Bu_j^{k,l})_{k,l}$ to 
$(\sum_{i=1}^{n_{k,l}}v_i^{k,l}u_j^{k,l})_{k,l}$  is isometric for all $n.$ 

Define $M=R_n(D).$ This is a TRO implementing strong TRO equivalence between $M_n(A)$ and $B.$ Since $R_n(U)= [BM]^{-\|\cdot\|}$ 
and $C_n(V)=[M^*B] ^{-\|\cdot\|} , \;\;\;M_n(A)=[C_n(V)R_n(U)]  ^{-\|\cdot\| }$ by the first part of the proof the map 
$$\rho : C_n(V)\otimes ^h_{ B} R_n(U) \rightarrow M_n(A)$$ sending every $v\otimes _Bu$ to $vu$
 is isometric and onto. By Proposition 1.5.14 in \cite{bm} the map 
$$\tau : C_n(V)\otimes ^h R_n(U) \rightarrow M_n(V\otimes ^hU)$$ given by 
$\tau (v\otimes u)=(v_i\otimes u_j)_{i,j}$ where $v=(v_1,...,v_n)^t, \;\;u=(u_1,...,u_n)$ is isometric. If 
$$\Omega =[vb\otimes u-v\otimes bu: b\in B, \;v\in C_n(V), \;u\in R_n(U)]^{-\|\cdot\|}$$ and 
$$\Xi  =[vb\otimes u-v\otimes bu: b\in B, \;v\in V, \;u\in U]^{-\|\cdot\|}$$ 
then we can consider $ C_n(V)\otimes _B^hR_n(U) =C_n(V)\otimes ^hR_n(U)/\Omega  $ and 
 $ V\otimes _B^hU =V\otimes ^hU/\Xi . $ We can see that $\tau (\Omega )=M_n(\Xi ),$ thus the map 
$$\hat{\tau }: C_n(V)\otimes ^h_B R_n(U) \rightarrow M_n(V\otimes ^hU)/M_n(\Xi )$$ 
sending every $v\otimes _Bu=v\otimes u+\Omega $ to $(v_i\otimes u_j)_{i,j}+M_n(\Xi )$ 
where $v=(v_1,...,v_n)^t, \;\;u=(u_1,...,u_n)$ is isometric surjection. 
Since the map $$\sigma : M_n(V\otimes ^hU)/M_n(\Xi )\rightarrow M_n(V\otimes ^hU/\Xi )=M_n(V\otimes _B^hU)$$ 
sending every $(\sum_{i=1}^{n_{k,l}}v_i^{k,l}\otimes u_j^{k,l})_{k,l}+M_n(\Xi )$ to  
$(\sum_{i=1}^{n_{k,l}}v_i^{k,l}\otimes _Bu_j^{k,l})_{k,l}$  is also isometric surjection we have that the map 
$$ \rho \circ \hat{\tau}^{-1} \circ \sigma ^{-1} : M_n(V\otimes _B^hU)\rightarrow M_n(A)$$ is isometric and onto.
We can easily see that $ id_n\otimes \theta  =\rho \circ \hat{\tau}^{-1} \circ \sigma ^{-1} ,$ thus 
$id_n\otimes \theta  $ is isometry.

We have proved that $\theta $ is completely isometric and onto. Similarly, we can prove that the 
spaces $B$ and $U\otimes _BV$ are completely isometrically isomorphic as $B-$modules. 
\end{proof}

In the sequel of  this section we are going to prove that if 
$A$ and $B$ are operator algebras with contractive approximate identities (cai's) 
 and are strongly $\Delta $-equivalent, then for every completely 
isometric representation $\alpha $ of $A$, there exists a completely isometric representation $\beta $ of $B$ such that 
$\alpha (A)$ and $\beta (B)$ are strongly TRO equivalent. We may assume that $A\subset B(R)$ and $B\subset B(L)$ for $R$ and $ L
$ some Hilbert spaces, and that there exists a norm closed TRO $M\subset B(R, L)$ 
such that $$A=[M^*BM ] ^{-\|\cdot\|}, \;\;\; B=[MAM^*]^{-\|\cdot\|} .$$ 
Let $$Y=[MA] ^{-\|\cdot\|} \;\; \text{and }\;\; X=[AM^*]^{-\|\cdot\|} .$$ We can easily see that 
$$Y=[BM]^{-\|\cdot\|} , \;\; \text{and }\;\;X=[M^*B]^{-\|\cdot\|} , $$ 
thus 
$$BYA\subset Y, \;\;\;AXB\subset X.$$ By Theorem \ref{24} and its proof, the algebra $A$ (resp. $B$) is completely isometrically 
isomorphic as an $A$-bimodule (resp. a $B$-bimodule) with the space $X\otimes _B^hY$ (resp. $Y\otimes _A^hX$). 
We assume that $\alpha:  A\rightarrow B(H)$ is a completely isometric representation such that $\overline{\alpha (A)(H)}=H. $  
We define the space $K=Y\otimes _A^hH$, which is the underlying Hilbert space of a representation of $B$, Theorem 3.10 in  
\cite{bmp}, through the following completely contractive map: $$\beta : B\rightarrow B(K),\;\;\; \beta (b)(y\otimes _Ah)=(by)\otimes _Ah.$$ 
We are going to prove that $\beta $ is a complete isometry and that the algebras $\alpha (A)$ and $\beta (B)$ are strongly TRO equivalent.

\begin{lemma}\label{25} Let $(f_\lambda )$ be the net from Lemma \ref{16}. Let
$$ \theta _\lambda  : K\rightarrow C_{k_\lambda }(H) $$ be the map defined by 
$$\theta _\lambda (y\otimes _Ah)=( \alpha ((n_1^\lambda )^* y)(h) , ...,\alpha ((n_{k_\lambda }^\lambda )^*y)(h) )^t.$$ 
If $\sca{\cdot, \cdot}_K$ is the inner product of $K$, then 
$$\sca{u,v}_K=\lim_\lambda \sca{\theta_\lambda (u),  \theta _\lambda (v)}_{C_{k_\lambda }(H)}\;\;\;\forall \;\;u,v \;\in 
K. $$
\end{lemma} 
\begin{proof} If $$u=\sum_{j=1}^my_j\otimes _Ah_j,$$ then 
\begin{align*}&  \|\theta _\lambda (u)\|= \|(\alpha ((n_i^\lambda)^*y_j ))_{i,j}(h_1,...,h_m)^t\| \leq 
\|(\alpha ((n_i^\lambda)^*y_j) )_{i,j}\|\|(h_1,...,h_m)^t\|\leq \\& \|((n_1^\lambda )^*,..., (n_{k_\lambda }^\lambda )^*)^t \|
\|(y_1,...,y_m)\|\|(h_1,...,h_m)^t\| \leq \|(y_1,...,y_m)\|\|(h_1,...,h_m)^t\| .
\end{align*}
We see that $\theta _\lambda $ is a contractive map. Fix  $a_1, ..., a_{k_\lambda }\in A, h_1,...,h_{k_\lambda }\in H.$ 
 If $(\hat{a_t})_t$ is a cai for $A$, then for any $\epsilon >0$ there exists $t$ such that 
\begin{align*} & \nor{\sum_{i=1}^{k_\lambda }n_i^\lambda a_i\otimes _Ah_i} -\epsilon \leq  
\nor{\sum_{i=1}^{k_\lambda }n_i^\lambda \hat{a_t}a_i\otimes _Ah_i} =\\& 
\nor{\sum_{i=1}^{k_\lambda }n_i^\lambda \hat{a_t}\otimes _A\alpha (a_i)(h_i)}\leq \nor{(\alpha (a_1)(h_1)
,...,\alpha (a _{k_\lambda } )
(h_{k_\lambda } ))^t} . 
\end{align*} 
Since $\epsilon $ was arbitrary,
$$\nor{\sum_{i=1}^{k_\lambda }n_i^\lambda a_i\otimes _Ah_i} \leq \nor{(\alpha (a_1)(h_1),...,\alpha (a_{k_\lambda })(h_{k_\lambda } ))^t} .$$
Therefore we can define a contraction $$\gamma _\lambda : C_{k_\lambda }(H) \rightarrow Y\otimes _AH$$ 
given by the type
$$\gamma _\lambda ((\alpha (a_1)(h_1),...,\alpha (a_{k_\lambda })(h_{k_\lambda }))^t)= \sum_{i=1}^{k_\lambda }
n_i^\lambda a_i\otimes _Ah_i , \;\;a_i\;\;\in A, \;\;\;h_i\;\;\in \;\;H . $$
If $m\in M$ and $ a\in A$, then
\begin{align*} & \gamma _\lambda \theta_ \lambda(ma\otimes _A h)=\gamma _\lambda((\alpha ((n_1^\lambda )^*ma)(h), ..., 
(\alpha (n_{k_\lambda })^*ma)(h))^t)=\\ & \sum_{i=1}^{k_\lambda } n_i^\lambda  (n_i^\lambda)^*ma\otimes _Ah=(f_\lambda ma)\otimes _Ah\stackrel
{\|\cdot\|}{\rightarrow }ma\otimes _Ah.  
\end{align*}
Since all $\gamma _\lambda \circ \theta_ \lambda $ are contractions and $Y=[MA]^{-\|\cdot\|} $, we have 
$$u= \|\cdot\|-\lim_\lambda  \gamma _\lambda \theta _\lambda (u)\;\;\;\forall \;\;u\;\;\in K.$$
We observe that $$\|u\|\geq \|\theta _\lambda (u)\|\geq \|\gamma _\lambda \theta _\lambda (u)\|.$$
 So $$\lim_\lambda \|\theta _\lambda (u)\|=\|u\|_K.$$
Thus  $$\sca{u,v}_K=\lim_\lambda \sca{\theta_\lambda (u),  \theta _\lambda (v)}_{C_{k_\lambda }(H)}\;\;\;\forall \;\;u,v \;\in 
K. $$ 

\end{proof}

\begin{lemma}\label{astron} For every $a, b\in A, c\in [M^*M]^{-\|\cdot\|}$, and 
$h,\xi \in H$, we have
$$ \sca{\alpha (a)(h), \alpha (cb)(\xi )} = \sca{\alpha (c^*a)(h), \alpha (b)(\xi )} .$$
\end{lemma}
\begin{proof} We denote the $C^*$-algebra by $C=[M^*M] ^{-\|\cdot\|} $ and by $ \cl M_l(A) $ the left multiplier 
algebra of $A$. Put 
$$\sigma : C\times A\rightarrow A, \;\; \sigma (c,a) =ca.$$
Since $A=[CA] ^{-\|\cdot\|} $ if $(c_t)$ is a cai for $C$, we have 
$$\lim_t\sigma (c_t,a)=\lim_tc_ta=a\;\;\;\forall \;\;a\;\;\in \;\;A.$$ So $\sigma $ is 
an oplication in the sense of Theorem 4.6.2 in \cite{bm}. Therefore, by that theorem, there exists a $*$-homomorphism 
$$\hat \theta : C\rightarrow \cl M_l(A)\cap \cl M_l(A)^*  ,\;\;\;\hat \theta (c)(a)=\sigma (c,a) =ca.$$ 
Let $\Omega $ be the algebra $$\{T\in B(H): T\alpha (A)\subset \alpha (A)\}.$$ 
By Theorem 2.6.2 in \cite{bm}, there exists a completely isometric homomorphism $$\rho : \Omega \rightarrow \cl M_l(A) : \;\;\;
\rho (T)(a)=\alpha ^{-1}(T\alpha (a)).$$
Put $$\theta =\rho ^{-1}\circ \hat \theta : C\rightarrow \Omega .$$ Since $$\hat \theta (c)(a)=ca\;\;\forall \;a\;\in \;A\Rightarrow 
\rho (\theta (c))(a)=ca\;\;\forall \;\;a\;\in \;A.$$ 
So $$\alpha (ca)=\alpha (\rho (\theta (c))(a))=\theta (c)\alpha (a)\;\;\;\forall \;c\;\in \;C, \;a\;\in \;A.$$ Since $\theta $ 
is a $*$-homomorphism,
\begin{align*}&\sca{\alpha (a)(h), \alpha (cb)(\xi )} =\sca{\alpha (a)(h), \theta (c
)\alpha (b)(\xi )} =\\&
\sca{\theta (c^*)\alpha (a)(h), \alpha (b)(\xi )}=\sca{\alpha (c^*a)(h), \alpha (b)(\xi )} .
\end{align*}
\end{proof}

\begin{lemma}\label{ploion1} The map $\phi : Y\rightarrow B(H, K)$ given by 
$\phi (y)(h)=y\otimes _Ah$ is a complete isometry.
\end{lemma}
\begin{proof} Clearly $\phi  $ is a completely contractive map.  It suffices to prove that 
$$\|y\|\leq \|\phi (y)\|$$ for arbitrary $y \in M_n(Y) $ and $ n\in \bb N$. 

Since 
 $Y=[MA]^{-\|\cdot\|}$, we need to show  $\|y\|\leq \|\phi (y)\|$ for $y=(y_{ij})\in M_n(Y)$, 
where $y_{ij}=m_{ij}a_{ij}$ with $m_{ij}\in R_k(M), a_{ij}\in C_k(A)$ and $ k \in \bb N$.   
There exist $s\in \bb N$, $m_i\in R_s(M)$, and $a_j\in C_s(A)$ 
such that $y_{ij}=m_{i}a_{j}$ for $1\leq i,j\leq n$. For example, if 
$$ \left(\begin{array}{clr} y_{11} & y_{12} \\ y_{21} & y_{22} \end{array}\right ) =
\left(\begin{array}{clr} m_{11}a_{11}  & m_{12}a_{12}  \\ m_{21}a_{21}  
& m_{22}a_{22}  \end{array}\right ),$$ then $y_{ij}=m_{i}a_{j}$ 
for the  rows $$ m_1=(m_{11}, \; 0,\; m_{12}, \; 0) , \;\;\;m_2=(0 ,\; m_{21}, \; 0,\; m_{22}) $$ 
and the columns $$a_1=(a_{11},\;a_{21}, \; 0,\; 0)^t, \;\;\;a_2=(0,\;0,\;a_{12},\;a_{22})^t.$$

Fix $h_1,...,h_n\in H$. We can see that $$ \|\phi (y)(h_1,...,h_n)^t\|^2 =\sum_{i=1}^n\nor{ 
\sum_{k=1}^ny_{ik}\otimes _Ah_k}_K^2.$$ We recall the maps $\theta _\lambda  $ from 
Lemma \ref{25}. We have
\begin{align*}& \|\phi (y)(h_1,...,h_n)^t\|^2 = \lim_\lambda \sum_{i=1}^n\sca{ \theta _\lambda 
(\sum_{k=1}^ny_{ik\otimes _Ah_k}), \theta _\lambda 
(\sum_{l=1}^ny_{il\otimes _Ah_l}) }= \\&  \lim_\lambda \sum_{i=1}^n \sum_{k=1}^n 
\sum_{l=1}^n \sca{ \theta _\lambda 
( m_ia_k\otimes _Ah_k), \theta _\lambda 
( m_ia_l\otimes _Ah_l) }=\\& \lim_\lambda \sum_{i=1}^n \sum_{k=1}^n 
\sum_{l=1}^n \sum_{j=1}^{k_\lambda }\sca{ \alpha 
( (n_j^\lambda )^*  m_ia_k)(h_k),  
\alpha ( (n_j^\lambda )^*  m_ia_l)(h_l) }.   
\end{align*}
By Lemma \ref{astron}, we have \begin{align*}\|\phi (y)(h_1,...,h_n)^t\|^2 =& 
\lim_\lambda \sum_{i=1}^n \sum_{k=1}^n 
\sum_{l=1}^n \sum_{j=1}^{k_\lambda }\sca{\alpha (m_i^* n_j^\lambda ( n_j^\lambda )^*  m_ia_k)(h_k),  
\alpha (a_l)(h_l)} =\\& \sum_{i=1}^n \sum_{k=1}^n 
\sum_{l=1}^n \sca{\alpha (m_i^*  m_ia_k)(h_k),  
\alpha (a_l)(h_l)} .\end{align*}
Again by Lemma \ref{astron}, we have 
\begin{align*}& \|\phi (y)(h_1,...,h_n)^t\|^2 = \sum_{i=1}^n \sum_{k=1}^n 
\sum_{l=1}^n \sca{\alpha ( (m_i^*  m_i)^{\frac{1}{2}} a_k )(h_k),  
\alpha ( (m_i^*  m_i)^{\frac{1}{2}} a_l)(h_l)}=\\& \sum_{i=1}^n\nor{\sum_{k=1}^n\alpha (
(m_i^*  m_i)^{\frac{1}{2}} a_k )(h_k)}^2 =\nor{\alpha ( ((m_i^*  m_i)^{\frac{1}{2}} a_k 
)_{i,k} ) (h_1,...,h_n)^t }^2.\end{align*}
 Taking the supremum over all $(h_1,...,h_n)^t$ with $\|(h_1,...,h_n)^t\|\leq 1$, we obtain 
$$\|\phi (y)\|^2=\|\alpha (((m_i^*  m_i)^{\frac{1}{2}} a_k 
)_{i,k} ) )\|^2.$$ Since $\alpha $ is a complete isometry,
$$\|\phi (y)\|^2=\|((m_i^*  m_i)^{\frac{1}{2}} a_k 
)_{i,k} ) \|^2=$$
$$\nor{(\sum_{k=1}^na_i^*m_k^*m_ka_j)_{i,j}}=\nor{(\sum_{k=1}^ny_{ki}^*y_{kj})}=\|y^*y\|=\|y\|^2.$$ 
The proof is complete. 

\end{proof}

\begin{lemma}\label{ploion2} If $b\in M_n(B)$ and $n\in \bb N$, then $$\|b\|=\sup_{y\in Ball(M_{n,k}(Y)), 
k\in \bb N}\|by\|.$$
\end{lemma}
\begin{proof} Suppose that $b=(b_{ij})$. Since $B$ is completely isometrically isomorphic 
as a $B$ bimodule to $Y\otimes _A^hX$, there exist nets $(y_k)_k, \;(x_k)_k,$ where $$y_k\in \; 
Ball(R_{n_k}(Y)) , x_k\in \; Ball(C_{n_k}(X)) $$ such that $$b_{ij}=\|\cdot\|-\lim_kb_{i
j}y_kx_k,$$ for all $i,j,$ Lemma 2.9 in \cite{bmp}. So for any $\epsilon >0$, there exists a $k$ such that 
\begin{align*} &\|b\|-\epsilon <\\ & \|(b_{ij}y_kx_k)_{i,j}\| =\|(b_{ij})_{ij} (y_k\oplus ...\oplus y_k) (
x_k\oplus ...\oplus x_k)\|\leq \|by\|, 
\end{align*}
where $y=(y_k\oplus ...\oplus y_k) $. Since $\epsilon $ was arbitrary, the proof is complete.
\end{proof}

\begin{lemma}\label{ploion3} The map $\beta $ is a complete isometry.
\end{lemma}
\begin{proof} Fix $b\in M_n(B)$ for some $ n\in \bb N$. By Lemmas \ref{ploion1} and \ref{ploion2}, 
we have 
\begin{align*} \|b\|=& \sup_{y\in Ball(M_{n,k}(Y)), 
k\in \bb N}\|by\| = \sup_{ y\in Ball(M_{n,k}(Y) ), 
k\in \bb N}\|\phi (by)\| =\\&  \sup_{y\in Ball(M_{n,k}(Y)), 
k\in \bb N}\sup_{\| (h_1,...,h_k)^t \|\leq 1}\| \phi (by) (h_1,...,h_k)^t  \| .
\end{align*} 
We can see that $$ \phi (by)(h_1,...,h_k)^t  =\beta (b)(y\otimes _A(h_1,...,h_k)^t  ).$$ 
So $$\|\phi (by)(h_1,...,h_k)^t \|\leq \|\beta (b)\| $$ 
for all $y\in Ball(M_{n,k}(Y) , h=(h_1,...,h_k)^t $ with $\|h\|\leq 1$. 

Thus $\|b\|\leq \|\beta (b)\|$. 
\end{proof}

Fix $a\in A$ and $ h\in H$. If $(a_t)_t$ is a cai for $A$ and $m\in M$, then 
$$\|ma\otimes _Ah\|=\lim_t\|ma_ta\otimes _Ah\|=\lim_t\|m\alpha _t\otimes _A\alpha (a)(h)\|.$$ 
So for any $\epsilon >0$, there exists $t$ such that $$\| ma\otimes _Ah \|-\epsilon \leq \|ma_t\otimes _A\alpha (a)(h)\|\leq 
\|m\|\|\alpha (a)(h)\|.$$ Since $\epsilon $ was arbitrary, we have  $$\| ma\otimes _Ah \| \leq  
\|m\|\|\alpha (a)(h)\|.$$ 
So we can define a map 
$$\alpha (A)(H)\rightarrow K: \alpha (a)(h)\rightarrow ma\otimes _Ah $$ since this map is bounded and 
$H=\overline{\alpha (A)(H)}$ extends to $$\mu (m): H\rightarrow K, \mu (m)(\alpha (a)(h))=ma\otimes _Ah.$$
We are going to prove that $N=\overline{\mu (M)}^{\|\cdot\|}$ is a TRO implementing a TRO equivalence between $\alpha (A)$  
and $\beta (B)$. 

Suppose that $m\in M, y_i\in Y$, and $ h_i\in H, i=1,...,k$; and let $(u_t)_t$ be the net in Lemma \ref{16}. We have 
$$ \|m\|\nor{\sum_{i=1}^ky_i\otimes _Ah_i}\geq  \nor{\sum_{i=1}^ku_tm^*y_i\otimes _Ah_i} =
\nor{\sum_{i=1}^k\alpha (u_tm^*y_i)(h_i)} .$$
Since $m^*=\|\cdot\|-\lim_tu_tm^*$, we have $$ \|m\|\nor{\sum_{i=1}^ky_i\otimes _Ah_i}\geq \nor{\sum_{i=1}^k\alpha (m^*y_i)(h)}.$$
Thus we can define a bounded map 
$$\nu (m^*): K\rightarrow H,\;\;\; y\otimes _Ah\rightarrow \alpha (m^*y)(h).$$ We are going to prove 
that $\mu (m)$ is the adjoint of $\nu (m^*)$.

\begin{lemma}\label{26} $$\nu (m^*)=\mu(m)^*\;\;\;\forall \;\;m\;\;\in \;\;M. $$
\end{lemma}
 \begin{proof} We recall the net $(f_\lambda )_\lambda $ and the maps 
$\theta _\lambda : K\rightarrow C_{k_\lambda }(H)$ from Lemma \ref{25}. For every 
$a, b\in A, r, m\in M$, and $h,\xi \in H$ we have
\begin{align*} & \sca{\mu (m)(\alpha (a)(h)), rb\otimes _A\xi }=\sca{ ma\otimes _Ah , rb\otimes _A\xi }=
\lim_\lambda \sca{\theta _\lambda(ma\otimes _ah ),  \theta_ \lambda (rb\otimes _A\xi) }=\\
& \lim_\lambda \sca{ (\alpha ((n_1^\lambda )^*ma)(h),...,\alpha ((n_{k_\lambda} ^\lambda )^*ma)(h))^t , 
(\alpha ((n_1^\lambda )^*rb)(\xi ),...,\alpha ((n_{k_\lambda} ^\lambda )^*rb)(\xi ))^t }=\\
& \lim_\lambda \sum_{j=1}^{k_\lambda } \sca{\alpha ((n_j^\lambda )^*ma )(h), \alpha ((n_j^\lambda )^*rb )(\xi )}.
\end{align*}
By Lemma \ref{astron},
\begin{align*} & \sca{\mu (m)(\alpha (a)(h)), rb\otimes _A\xi }=
\lim_\lambda \sum_{j=1}^{k_\lambda } \sca{\alpha (r^*n_j^\lambda (n_j^\lambda )^*ma )(h), \alpha (b )(\xi )}=\\
& \lim_\lambda \sca{\alpha (r^*f_\lambda ma)(h), \alpha (b)(\xi )}=\sca{\alpha (r^*ma)(h), \alpha (b)(\xi )}=\\
& \sca{\alpha (a)(h), \alpha (m^*rb)(\xi )}=\sca{\alpha (a)(h), \nu (m^*)(rb\otimes _A\xi )}.
\end{align*}
Since $\alpha (A)(H)$ is dense in $H$ and $Y=[MA]^{-\|\cdot\|}$, the proof is complete.
\end{proof}

 \begin{theorem}\label{27} Suppose that $A$ and $B$ are operator algebras with contractive 
approximate identities which are strongly $\Delta $-equivalent. 
Then for every completely isometric 
representation $\alpha $ of $A$, there exists a completely isometric 
representation $\beta $ of $B$ such that $\alpha (A)$ and $\beta (B)$ are strongly TRO equivalent.
\end{theorem}
\begin{proof} We assume that $A, B$, and $M$ are as above. We also recall the maps $\alpha , \beta ,\mu 
$, and $ \nu $.  By Lemma \ref{ploion3}, $\beta $ is a complete isometry. 
If $N=\overline{\mu (M)}^{\|\cdot\|}$,
we are going to prove that $N$ is a TRO and 
$$\alpha (A)=[N^*\beta (B)N] ^{-\|\cdot\|} , \;\;\;\beta (B)=[N\alpha (A)N^*]^{-\|\cdot\|} .$$
If $m_1, m_2, m_3 \in M, a\in A,$ and $h\in H$, we have 
$$\mu (m_3)\mu(m_2)^* \mu (m_1)(\alpha (a)(h))=\mu (m_3)\nu (m_2^*)(m_1a\otimes _Ah)=\mu (m_3)(\alpha (m_2^*m_1a)(h))=$$
$$m_3m_2^*m_1a\otimes _Ah=\mu (m_1m_2^*m_3)(\alpha (a)(h)).$$
So $$\mu(m_3) \mu (m_2)^*\mu (m_1)=\mu (m_3m_2^*m_1)\in \mu (M)\subset N.$$
Thus $$NN^*N\subset N.$$ 
If $m_1, m_2\in M, b\in B, a\in A$, and $h\in H$, we have 
$$ \mu (m_2)^*\beta (b)\mu (m_1) (\alpha (a)(h))=\nu (m_2^*)\beta (b)(m_1a\otimes _Ah)=$$
$$\nu (m_2^*)(bm_1a\otimes _Ah)=\alpha (m_2^*bm_1a)(h)=\alpha (m_2^*bm_1)\alpha (a)(h).$$
So $$\mu (m_2)^*\beta (b)\mu (m_1) =\alpha (m_2^*bm_1).$$ Since $\alpha $ and $\beta $ are completely isometric maps and 
$A=[M^*BM] ^{-\|\cdot\|}$, then $\alpha (A)=[N^*\beta (B)N] ^{-\|\cdot\|}.$
If additionally $y\in Y$, then
$$ \mu(m_2)\alpha (a )\mu (m_1)^*  (y\otimes _Ah)=\mu (m_2)\alpha (a)\nu(m_1^*) (y\otimes _Ah)=$$
$$\mu (m_2)(\alpha (am_1^*y)(h))=m_2am_1^*y\otimes _A h= \beta (m_2am_1^*) (y\otimes _Ah).$$
Thus $$\mu(m_2)\alpha (a )\mu (m_1)^*  =\beta (m_2am_1^*) .$$
Since  $$B=[MAM^*]^{-\|\cdot\|}\Rightarrow \beta (B)=[N\alpha (A)N^*]^{-\|\cdot\|} .$$
The proof is complete.
\end{proof}

\begin{corollary}\label{28} Strong $\Delta $-equivalence is an equivalence relation of operator 
algebras with contractive 
approximate identities.
\end{corollary}
\begin{proof} We need to prove its transitivity. Suppose that $A$, $B$, and $C$ are operator algebras with 
contractive approximate identities and 
that $A$ and $B$ (resp. $B$ and $C$) are strongly $\Delta $-equivalent. 
By Definition \ref{15}, there exist completely isometric representations $\alpha $ of $A$ and $\beta $ of $B$ 
such that $\alpha (A)$ and $\beta (B)$ are strongly TRO equivalent. By Theorem \ref{27}, there exists a completely isometric representation $\gamma $ 
of $C$ such that the algebras $\beta(B)$ and $ \gamma (C)$ are strongly TRO equivalent. By Theorem \ref{21}, 
the algebras $\alpha (A)$ and $\gamma (C)$ are strongly TRO equivalent.   
\end{proof}

\section{Stable isomorphisms of operator algebras}

If $X$ is an operator space, $M_\infty (X)$ denotes the operator space of $\infty \times \infty $ matrices with entries in 
$X,$ whose finite submatrices have uniformly bounded norm. 
 Let $M_\infty ^{fin}(X)$ denote the subspace of finitely supported matrices and write $K_\infty (X)$ 
for its norm closure in  $M_\infty (X)$. We can see that $K_\infty (X)$ is isomorphic as an operator space with 
$X\otimes \cl K$, where $\otimes $ is the spatial tensor product  and $\cl K$ is the algebra of compact operators 
acting on an infinite dimensional separable Hilbert space. 

Suppose that $X$ and $Y$ are operator spaces. We call them strongly stably isomorphic if   $K_\infty (X)$ and $K_\infty (Y)$  
are isomorphic as operator spaces. In this section we are going to generalise, to the setting of nonselfadjoint operator algebras,
 the following very important theorem from \cite{bgr}:

\begin{theorem}\label{31} Two $C^*$-algebras which possess countable approximate identities are strongly Morita equivalent iff they are strongly 
stably isomorphic. 
\end{theorem} 
 
Our generalisation states:

\begin{theorem}\label{32} If two operator algebras which possess  countable approximate 
identities are strongly $\Delta $-equivalent then they are strongly 
stably isomorphic. Conversely, if two operator algebras which possess contractive approximate 
identities are strongly 
stably isomorphic then they are strongly $\Delta $-equivalent. \end{theorem}

The one direction of the proof is a consequence of the results of Section 2.
 We use Corollary \ref{28}: suppose $A$ and $B$ are operator algebras with 
contractive approximate 
identities such that 
$K_\infty (A)$ and $K_\infty (B)$ are isomorphic as operator spaces. (We recall that $C^*-$ algebras 
have  contractive approximate 
identities). We may assume that $A$ acts on the Hilbert space $H$ 
and $B$ acts on $L$. We can see that $$K_\infty (A)= [M^*AM]^{-\|\cdot\|} , \;\;A=MK_\infty (A)M^*,$$ 
where $M$ is the norm closure of finitely supported rows with scalar entries. 
Thus $A$ and $K_\infty (A)$ are strongly TRO equivalent. Since also  $K_\infty (B)$ and 
$B$ are strongly TRO equivalent and  $K_\infty (A)$ and $K_\infty (B)$ are isomorphic, we conclude that 
$A$ and $B$ are strongly $\Delta $-equivalent. For this direction we didn't use  the hypothesis  of the existence 
of a countable approximate identity. For the converse, we use this assumption. Examples in 
\cite{bgr} show that the hypothesis that 
the $C^*-$algebras have countable approximate units (equivalently, strictly positive elements) is not superfluous in the strong stable isomorphism theorem.

For the proof of Theorem \ref{32}, we fix operator algebras $A$ and  $B$ acting on 
the Hilbert spaces $H$ and $K$, respectively,  such that $A(H)$ (resp. $B(K)$) is dense in $H$ 
(resp. $K$) and which possess 
countable approximate identities. We also assume that 
there exists a norm closed TRO $M\subset B(H, K)$ such that $$A=[M^*BM] ^{-\|\cdot\|} ,\;\;\;, B=[MAM^*]^{-\|\cdot\|}.$$ 
We are going to prove that $K_\infty (A)$ and $K_\infty (B)$ are isomorphic as operator spaces. 
We define the spaces $$Y=[MA]^{-\|\cdot\|}=[BM]^{-\|\cdot\|}, \;\;\;X=[AM^*]^{-\|\cdot\|}=[M^*B]^{-\|\cdot\|}. $$
Also observe that $$A=[M^*MAM^*M]^{-\|\cdot\|}.$$ 
We define the $C^*$-algebra $$D=[\Pi _{i=1}^kA_iB_i, \;A_i=A^*, \;B_i=A, \;k\in \bb N]^{-\|\cdot\|} .$$ 

\begin{lemma}\label{positive} There exists an element $a_0 \in D$ such that $D=\overline{Da_0}^{\|\cdot\|}.$
\end{lemma}
\begin{proof} It suffices to prove that $D$ has a strictly positive element. Suppose that $(e_n)_{n\in \bb N}$ 
is an approximate identity for $A.$ Define $$a_0=\sum_{n=1}^\infty \frac{e_n^*e_n}{\|e_n\|^2 2^n}$$ 
and fix a state $\phi $ of $D.$ We are going to prove that $\phi (a_0)>0.$ If, on the contrary, $\phi (a_0)=0,$ 
then $\phi (e_n^*e_n)=0$ for all $n. $ Fix an arbitrary $d\in D $ and $a,b \in Ball(A).$ Since 
$a^*be_n\in A^*AA\subset D$, we have 
$$| \phi (da^*be_n) |^2\leq \phi(dd^*) \phi (e_n^*b^*aa^*be_n).$$
But $$0\leq e_n^*b^*aa^*be_n\leq e_n^*e_n.$$ Thus 
$$\phi (e_n^*b^*aa^*be_n)=0\Rightarrow \phi (da^*be_n) =0\;\;\forall \;n.$$
The sequence $(be_n)_n$ converges to $b.$ We conclude that $\phi (da^*b)=0$ for all $d\in D, \;a,b\in A,$ 
which implies $\phi (\Pi  _{i=1}^k a_i^*b_i)=0$ for all $a_1, ..., a_k, b_1,...,b_k \in A,\; k\in \bb N.$ 
It follows that $\phi =0.$ This contradiction completes the proof.
\end{proof}

\begin{lemma}\label{35} There exists a sequence $(m_i)_{i\in \bb N}\subset M$ such that 
$$\nor{ \sum_{i=1}^km_i^*m_i }\leq 1, \;\;\;\forall \;\;k\;\;\in \;\bb N$$ and 
$$\|\cdot\|-\lim_kd\sum_{i=1}^km_i^*m_i=d\;\;\;\forall \;\;d\;\;\in \;D. $$
\end{lemma}
\begin{proof} The proof is similar to that of Lemma 2.3 of \cite{brown}. By Lemma \ref{16}, there exists a net $(u_t)_t$ 
where $$u_t=\sum_{i=1}^{l_t}(r_i^t)^*r_i^t, \;\;\;r_i^t\;\;\in \;\;M\;\;\forall \;\;i, t$$
 such that $$0\leq u_t\leq I_H, \;\;\;\|\cdot\|-\lim_tu_tm^*=m^*\;\;\forall \;\;m\;\in \;M.$$
Since  $$D=[M^*MDM^*M] ^{-\|\cdot\|},$$ we have $$\|\cdot\|-\lim_tu_td=d
\;\;\forall \;\;d\;\in \;D.$$ Thus, there exists $t_1$ such that $$\|(I_H-u_{t_1})a_0\|<1.$$ 
We write $$u_{t_1}=\sum_{i=1}^{k_1}m_i^*m_i,$$ where 
$$m_i=r_i^{t_1}, \;\;k_1=l_{t_1}.$$ Therefore  
$$\nor{(I_H-  \sum_{i=1}^{k_1}m_i^*m_i )a_0}<1.$$ 
Suppose that we have found integers $k_1<k_2<...<k_{n-1}$ such that 
$$0\leq  \sum_{i=1}^{k_l}m_i^*m_i  \leq I_H, \;\;\;\nor{(I_H- \sum_{i=1}^{k_l}m_i^*m_i  )a_0}<\frac{1}{l}$$
for every $l\in \{1,...,n-1\}$. Write $$s=\sum_{i=1}^{ k_{n-1} }m_i^*m_i  .$$ 
We have 
$$ (I_H-s)^{\frac{1}{2}} (I_H-u_t)(I_H-s)^{\frac{1}{2}} a_0=(I_H-s)a_0-(I_H-s)^{\frac{1}{2}}u_t (I_H-s)^{\frac{1}{2}}a_0.$$ 
Since $(I_H-s)^{\frac{1}{2}}a_0\in D,$ the above net converges to $0$. So there exists $u_{t_n}$ such that 
$$\nor{(I_H-s)^{\frac{1}{2}} (I_H-u_{t_n})(I_H-s)^{\frac{1}{2}}a_0 }<\frac{1}{n}.$$
Suppose that $$u_{t_n}=\sum_{i=1}^lr_i^*r_i$$ and put 
$$ m_{ k_{n-1} +1 }=r_1 ( I_H-s)^{\frac{1}{2}}  , ..., m_{k_{n}}=r_l(I_H-s)^{\frac{1}{2}} ,$$ where $k_n=l+k_{n-1}$. 

We can see that 
\begin{align*} & \nor{(I_H- \sum_{i=1}^{k_n}m_i^*m_i )a_0}= \nor{(I_H-s-\sum_{i=k_{n-1} +1 }^{k_n}m_i^*m_i)a_0} =\\&
\nor{(I_H-s- ( I_H-s)^{\frac{1}{2}}  \sum_{i=1}^{l}r_i^*r_i( I_H-s)^{\frac{1}{2}}  )a_0} =\\
& \nor{( I_H-s)^{\frac{1}{2}}  (I_H-u_{t_n})( I_H-s)^{\frac{1}{2}}  a_0}<\frac{1}{n}.
\end{align*} 
We also see that  
$$ 0\leq  \sum_{i=1}^{k_n}m_i^*m_i  =s+( I_H-s)^{\frac{1}{2}}  \sum_{i=1}^lr_i^*r_i( I_H-s)^{\frac{1}{2}}  \leq s+I_H-s=I_H.$$
Therefore, there exist operators 
$$\{m_i: 1\leq i\leq k_n\}_n\subset M$$ such that 
$$0\leq  \sum_{i=1}^{k_n}m_i^*m_i  \leq I_H, \;\; \nor{(I_H-\sum_{i=1}^{k_n}m_i^*m_i  )a_0}<\frac{1}{n}\;\;\;\forall \;\;n.$$
We conclude that $$\|a_0^*(I_H-\sum_{i=1}^{k_n}m_i^*m_i )a_0 \|\rightarrow 0.$$ 
Since the sequence $\|a_0^*(I_H-\sum_{i=1}^{n}m_i^*m_i )a_0 \|$ is decreasing, we have
$$\|a_0^* (I_H-\sum_{i=1}^{n}m_i^*m_i )a_0  \|\rightarrow 0.$$ The inequality 
$$0\leq  a_0^*(I_H-\sum_{i=1}^{n}m_i^*m_i )^2a_0  \leq a_0^*(I_H-\sum_{i=1}^{n}m_i^*m_i )a_0  $$
 implies that
$$\lim_n\nor{a_0^*(I_H-\sum_{i=1}^{n}m_i^*m_i  )^2a_0 }=0.$$
It follows that  $$ \|\cdot\|-\lim_k\sum_{i=1}^km_i^*m_ia_0=a_0 \Rightarrow 
 \|\cdot\|-\lim_ka_0\sum_{i=1}^km_i^*m_i=a_0 .$$ 
Since  by Lemma \ref{positive}  $D=\overline{Da_0}^{\|\cdot\|}$,
  we have $$\|\cdot\|-\lim_kd\sum_{i=1}^km_i^*m_i=d\;\;\;\forall \;\;d\;\;\in \;D. $$\end{proof}

\begin{lemma}\label{36} Let $(m_i)_{i\in \bb N}$ be the sequence in Lemma \ref{35}. Then 
$$ \|\cdot\|-\lim_ka\sum_{i=1}^km_i^*m_i=a\;\;\;\forall \;\;a\;\;\in \;A. $$
\end{lemma}
\begin{proof} Fix $a\in A$ and suppose that $a=u|a|$ is the polar decomposition of $a$. 
Since 
$|a|=(a^*a)^{\frac{1}{2}}$ and $A^*A\subset D$, we have $|a|\in D.$ Lemma \ref{35} gives 
\begin{align*}& \|\cdot\|-\lim_k|a|\sum_{i=1}^km_i^*m_i=|a|\Rightarrow 
\|\cdot\|-\lim_ku|a|\sum_{i=1}^km_i^*m_i=u|a|\Rightarrow \\& 
\|\cdot\|-\lim_ka\sum_{i=1}^km_i^*m_i=a.
\end{align*}
\end{proof}

We will use the following notation.

If $Z$ is a norm closed subspace of $B(L, R)$, where $L$ and $R$ are Hilbert spaces, we denote by $C_\infty (Z)$ the subspace of $B(L, R^\infty )$ 
containing all operators of the form $ (z_1,z_2,...)^t $ such that $z_i\in Z, \forall i$ and such that the sequence $(\sum_{i=1}^nz_i^*z_i)_n$ 
converges in norm. Similarly, $R_\infty (Z)$ is the subspace of $B(L^\infty , R)$ 
containing all operators of the form $(z_1,z_2,...)$ such that $z_i\in Z, \forall i$ and such that the sequence $(\sum_{i=1}^nz_iz_i^*)_n$  
converges in norm. 

If two operator spaces $Z_1, Z_2$ are completely isometrically isomorphic, we write $Z_1\cong Z_2.$

If $Z_i\subset B(L_i,R), i=1,2$ we denote by $Z_1\oplus _rZ_2$ the space 
$$\{(z_1, z_2): L_1\oplus L_2\rightarrow R\}.$$ If $Z_i\subset B(L_i,R), i\in \bb N$ is a sequence of norm closed spaces,
we  denote by $$Z_1\oplus _rZ_2\oplus _r...$$ the space of operators of the form 
$$(z_1, z_2,...): \oplus _{i=1}^\infty L_i\rightarrow R, \;\;\;z_i\;\in \;Z_i,\;\; i\in \;\bb N$$ 
such that the sequence  $(\sum_{i=1}^nz_iz_i^*)_n$ 
converges in norm.

\medskip

We now return to the proof of Theorem \ref{32}. Let $A, B, M, X$, and $Y$ be as in the discussion preceding Lemma 
\ref{positive}, 
and let $(m_i)_{i\in \bb N}$ be the sequence in Lemma \ref{36}. 
Put $$\alpha : Y\rightarrow R_\infty (B), a(y)=(ym_i^*)_i\;\;\;\beta : R_\infty (B)\rightarrow Y, 
\beta ((b_i)_i)= \sum_{i=1}^\infty  b_im_i.$$
These maps are completely contractive. Since $Y=[MA]^{-\|\cdot\|}$, by Lemma \ref{36} 
we have $$ \|\cdot\|-\lim_ky\sum_{i=1}^km_i^*m_i=y\;\;\;\forall \;\;y\;\;\in \;Y. $$ 
Thus $$\beta \circ \alpha (y)=\sum_{i=1}^\infty ym_i^*m_i=y. $$ 
We conclude that $\alpha $ is completely isometric. Put 
$$P=\alpha \circ \beta : R_\infty (B)\rightarrow R_\infty (B).$$
 We can see that $P$ is an idempotent map. 

If $b,c \in R_\infty (B)$, then
\begin{equation}\label{ex} P(b)c^*=\sum_{i,k}b_im_im_k^*c_k^*=bP(c)^*.\end{equation} 
We claim that $$R_\infty (B)\cong Ran P\oplus _r Ran (id-P).$$ 
Indeed, if $b\in R_\infty (B)$, then by using (\ref{ex}) we have 
\begin{align*}& \nor{(P(b) , P^\bot (b))}^2=\|P(b)P(b)^*+P^\bot(b) P^\bot (b)^*\|=\\
& \|bP(b)^*+bP^\bot (b)^*\|=\|bb^*\|=\|b\|^2.
\end{align*}
So the above map is isometric. Similarly, we can prove that it is completely isometric.  
Also, if $y\in Y$ and $b\in Ran(id-P)$, then
\begin{align*} \nor{ (y,  b) }^2=& \|yy^*+bb^*\|=\nor{\sum_{i=1}^\infty ym_i^*m_iy^*+bb^*}=
\\& \|\alpha(y) \alpha (y)^*+bb^*\|=
\nor{ (\alpha (y),  b) }^2. 
\end{align*}
So the map $$Y\oplus _rRan(id-P)\rightarrow \alpha (Y) \oplus _rRan(id-P): \;\;\;(y, b) \rightarrow (\alpha (y), b)  $$ 
is isometric. Similarly, we can prove that it is completely isometric.  
Thus, since $\alpha (Y)=Ran P$ if $W=Ran(id-P)$, we have $$R_\infty (B)\cong Y\oplus _rW.$$
Now we have 
\begin{align*} & R_\infty (B)\cong R_\infty (R_\infty(B))\cong  (Y\oplus  _rW)\oplus _r(Y\oplus _rW)\oplus ...
\cong  \\ & Y\oplus  _r(W\oplus _rY)\oplus _r ...\cong Y\oplus _r R_\infty (B).\end{align*} 

Therefore 
\begin{align*} & R_\infty (B)\cong R_\infty(R_ \infty(B))\cong R_\infty (Y\oplus _rR_\infty (B))\cong\\&  
R_\infty (Y)\oplus _rR_\infty(B).  
\end{align*}
Using Lemma \ref{16} and repeating the above arguments, we can find a sequence $(n_i)_i\subset M$ such that 
$$0\leq \sum_{i=1}^kn_in_i^*\leq I_K$$
 and $$b=\|\cdot\|-\lim_kb\sum_{i=1}^kn_in_i^*\;\;\;\forall \;b\;\in \;B.$$ 
Define the completely contractive maps 
$$\phi : B\rightarrow R_\infty (Y), \;\;\phi (b)=(bn_i)_i$$
$$\psi : R_\infty (Y)\rightarrow B, \;\;\psi ((y_i)_i)=\sum_{i}y_in_i^*.$$
Observe that $$\psi \circ \phi (b)=b\;\;\;\forall\;\; b\;\;\in B.$$ 
As before, we can prove that $$R_\infty (Y)\cong R_\infty(B)\oplus _rR_ \infty(Y). $$
Thus 
$$ R_\infty (B)\cong R_\infty (Y) \Rightarrow C_\infty (R_\infty (B))\cong C_\infty (R_\infty (Y) )\Rightarrow 
K_\infty (B)\cong K_\infty (Y).$$ 
Using the same methods, we can prove $$ C_\infty (Y)\cong C_\infty (A) \Rightarrow R_\infty (C_\infty (Y))
\cong R_\infty (C_\infty (A) )\Rightarrow 
K_\infty (Y)\cong K_\infty (A).$$
 
We can then conclude that $K_\infty (A)$ and $K_\infty (B)$ are isomorphic as operator spaces. 
The proof of Theorem \ref{32} is complete. 

\begin{theorem}\label{37} Strong Morita equivalence in the sense of Blecher, Muhly and Paulsen is strictly weaker 
than strong $\Delta $-equivalence. 
\end{theorem}
\begin{proof} There exists an example of  strongly  Morita equivalent, in the sense of 
Blecher, Muhly and Paulsen,   
operator algebras with unit of norm $1$ which are not stably isomorphic 
(Example 8.2 in \cite{bmp}). So, by Theorem 
\ref{32},  these algebras can not be strongly $\Delta $-equivalent.
\end{proof}

\begin{example}\em{Let $A$ and $B$ be nest  algebras corresponding to the nests  $\cl L_1$ and $\cl L_2$, 
acting on the separable Hilbert spaces 
$H$ and $K$, respectively. See the appropriate definition in \cite{dav}.  We assume that $\cl K(A)$ and $\cl K(B)$ are the subalgebras 
of compact operators. The second duals of $\cl K(A)$ and $ \cl K(B)$ are 
the algebras $A$ and $B$. Then the following are equivalent:

(i) $\cl K(A)$ and $\cl K(B)$ are strongly stably isomorphic.

(ii) $A$ and $B$ are weakly stably isomorphic.

(iii) There exists a $*$-isomorphism $\theta : \cl L_1^{\prime \prime} \rightarrow \cl L_2
^{\prime \prime} $ mapping $\cl L_1$ onto $\cl L_2$. Here, $\cl L_i^{\prime \prime}$ 
 is the double commutant of $\cl L_i, i=1,2$. 

The equivalence of (ii) and (iii) is implied by Theorems 3.3 in \cite{ele1} and 3.2 in \cite{ele3}. 

We shall prove that (i) implies (ii). We assume that $\cl K$ is the algebra of compact operators acting on the 
infinite dimensional separable Hilbert space $R$. Since $\cl K(A)\otimes \cl K$ and 
$\cl K(B)\otimes \cl K$ are isomorphic operator algebras, their second duals 
$A\otimes^\sigma  B(R)$ and $B\otimes^\sigma  B(R)$ are isomorphic as dual operator algebras. 
Here $\otimes $ is the spatial tensor product and  $\otimes^\sigma  $ is the normal 
spatial tensor product. 

We shall prove that (iii) implies (i). We define the TRO 
$$M=\{m\in B(H, K): mp=\theta (p)m \;\;\forall \;p\;\in \;\cl L_1\}.$$ 
By Theorem 3.3 in \cite{ele1},
$$ A=[M^*BM]^{-w^*} , \;\;\;B=[MAM^*]^{-w^*} .$$
Thus $$\cl K(A)\supset M^*\cl K(B) M, \;\;\;\cl K(B)\supset M\cl K(A)M^*.$$
On the other hand,
\begin{equation}\label{ee} M^*M\cl K(A)M^*M\subset M^*\cl K(B)M.
\end{equation} 
 By Theorem 8.5.23 in \cite{bm}, there exists a net of integers $(n_i)$ and 
operators $m_i\in Ball(C_{n_i}(M))\;\;\forall \;\;i$ such that the identity operator 
of $H$ is the limit of the net $m_i^*m_i$ in the strong operator topology. Thus $$k=\|\cdot\|-
\lim_im_i^*m_ik$$ for every compact operator $k\in B(H). $ It follows from (\ref{ee}) 
that  $$ \cl K(A)\subset [M^*\cl K(B)M]^{-\|\cdot\|}  \Rightarrow 
\cl K(A)= [M^*\cl K(B)M]^{-\|\cdot\|}.$$ Similarly we can prove that 
 $$\cl K(B)= [M\cl K(A)M^*]^{-\|\cdot\|}.$$ 
Since $A$ and $B$ are nest algebras acting on separable Hilbert spaces,  $\cl K(A)$ and $\cl K(B)$ 
have countable approximate identities, \cite{dav}. So by Theorem \ref{32},  $\cl K(A)$ and $\cl K(B)$ 
are strongly stably isomorphic.}
\end{example}

\end{document}